\documentclass[a4paper,10pt]{amsart}
\usepackage{amsmath,amsthm,amssymb,mathtools}
\usepackage[T1]{fontenc}
\usepackage[utf8]{inputenc}

\newtheorem{theorem}{Theorem}[section]
\newtheorem{proposition}[theorem]{Proposition}

\newtheorem{lemma}[theorem]{Lemma}
\theoremstyle{definition}
\newtheorem{remark}[theorem]{Remark}

\DeclareMathOperator{\sym}{Sym}
\DeclareMathOperator{\alt}{Alt}
\DeclareMathOperator{\Int}{Int}
\DeclareMathOperator{\car}{char}
\newcommand{\id}{\mathrm{id}}

\title[Artin--Springer theorem for generalized quadratic forms]{An elementary proof of the Artin--Springer theorem for generalized quadratic forms over quaternion algebras}
\author{A.-H. Nokhodkar}
\address{Department of Pure Mathematics, Faculty of Mathematical Sciences, 
University of Kashan, Kashan, Iran.}
\email{a.nokhodkar@kashanu.ac.ir}

\begin{document}
\maketitle

\begin{abstract}
We give a short and simple proof that an
anisotropic generalized quadratic form over a quaternion division algebra
with involution of the first kind, in arbitrary characteristic, remains anisotropic after
scalar extension to any odd-degree field extension.
\end{abstract}

%======================================================================
\section{Introduction}

A classical theorem due to Artin and Springer \cite{springer} asserts that
in characteristic different from $2$, an anisotropic quadratic form remains
anisotropic under any odd-degree field extension of the base field (see
\cite[Cor.~18.5]{ekm} for a characteristic-free treatment). From the
viewpoint of algebraic groups, it is natural to ask whether this extends to
hermitian forms with values in a division algebra $D$ equipped with an
involution $\theta$ of the first kind. This question is open in general;
the answer is known only in certain specific cases (see \cite{bqm} for a
survey). In particular, when $F$ has characteristic different from $2$ and
$D$ is a quaternion division algebra, the result holds by a theorem of
Parimala, Sridharan, and Suresh \cite{pss}. The proof uses an excellence
property of the corresponding generic splitting field, namely the function
field of a conic.

In contrast to the situation in characteristic different from $2$,
hermitian forms do not suffice to describe twisted forms of orthogonal
groups. One must instead resort to the finer notion of generalized
quadratic forms, which, via adjunction, gives rise to the so-called
quadratic pairs \cite{et}. From the viewpoint of algebraic groups, it appears more
pertinent to study the behavior of these generalized quadratic forms under
field extensions of odd degree. The argument of \cite{pss} also works for generalized quadratic forms over quaternion algebras in
characteristic~$2$, combining \cite{mt} and \cite{bd} as explained in
\cite{qt}. We
therefore have the following theorem.

\begin{theorem}\label{main}
Let $(Q,\theta)$ be a
quaternion division algebra with involution of the first kind over a field $F$ of arbitrary characteristic
and let $(V,\rho)$ be a generalized quadratic space over $(Q,\theta)$. 
Let $L/F$ be a finite field extension of odd degree. If $\rho_L$ is
isotropic, then $\rho$ is isotropic.
\end{theorem}

Recently, Qu\'eguiner-Mathieu and Tignol \cite{qt} gave a
characteristic-free proof avoiding the excellence property: realizing the
function field of the conic as a rational function field $K(t)$, they
transport $\rho$ to an ordinary quadratic form via a Morita equivalence on
a split model of $Q$, and run a degree-descent argument on a
crossed-product algebra, in the spirit of Rost's proof \cite{rost} of
excellence for conics.

In this paper, we give a more direct and elementary proof of Theorem~\ref{main}. Unlike
\cite{qt}, we avoid the generic splitting field of $Q$, the associated
Morita equivalence and the Laurent-degree filtration. Instead, a fixed quaternion basis of
$Q$ allows us to decompose $\rho$ directly over $F$ into an $F$-valued
quadratic form $\phi$ and a $K$-valued quadratic form $\psi$, where $K$ is
a separable quadratic extension of $F$. Applying the Amer--Brumer theorem
together with Springer's theorem over $F(t)$, we then reduce
isotropy of $\rho_L$ to a nonzero solution of $\phi(x)+t\psi(x)=0$ over
$F(t)$. The latter equation is handled by an explicit polynomial descent
in $Q[t]$, based on a degree-lowering transformation.

\section{Preliminaries and notations}
\label{sec:p}

Let $A$ be a central simple algebra over a field $F$.
An \emph{involution} on $A$ is an anti-automorphism of $A$ of order two.
An involution $\sigma$ on $A$ is said to be of \emph{the first kind} if it
restricts to the identity on $F$.
Otherwise, it is said to be of the \emph{second kind}.
An involution $\sigma$ of the first kind is \emph{symplectic} if it becomes
adjoint to an alternating bilinear form over a splitting field of $A$, and
\emph{orthogonal} otherwise.

For an algebra with involution $(A,\sigma)$ and $\lambda=\pm1$ we use the
notation
\begin{align*}
\sym_\lambda(A,\sigma)&=\{x\in A\mid \sigma(x)=\lambda x\},\\
\alt_\lambda(A,\sigma)&=\{x-\lambda\sigma(x)\mid x\in A\}.
\end{align*}
Note that if $\car F\neq2$ then $\sym_\lambda(A,\sigma)=\alt_{-\lambda}(A,\sigma)$
(see \cite[p.~14]{knus}).
We will simply denote $\sym_1(A,\sigma)$ by $\sym(A,\sigma)$ and
$\alt_1(A,\sigma)$ by $\alt(A,\sigma)$.

\medskip

Let \(D\) be a division algebra over \(F\) endowed with an involution \(\theta\) of the first kind, and let \(V\) be a finite-dimensional right \(D\)-vector space. For \(\lambda=\pm1\), a \emph{\(\lambda\)-hermitian space} over \((D,\theta)\) is a pair \((V,h)\) in which \(h:V\times V\to D\) is a bi-additive map satisfying
\[
h(ud,vd')=\theta(d)h(u,v)d', \qquad h(v,u)=\lambda\theta(h(u,v))
\]
for every \(u,v\in V\) and \(d,d'\in D\). Consequently, \(h(v,v)\in\sym_\lambda(D,\theta)\) for all \(v\in V\).

Write $\overline D=D/\alt_\lambda(D,\theta)$ and $\bar d=d+\alt_\lambda(D,\theta)$
for $d\in D$.
A map $\rho\colon V\to\overline D$ is called a \emph{generalized quadratic form (of type $\lambda$)} over
$(D,\theta)$ if 
\begin{equation}\label{eq1}
\rho(x+y)-\rho(x)-\rho(y)=\overline{h(x,y)}
\quad(x,y\in V)
\end{equation}
and $\rho(xa)=\theta(a)\rho(x)a$ for all $x\in V$ and $a\in D$, where $h$ is a $\lambda$-hermitian form over $(D,\theta)$.
Note that if $D=F$, then $\rho$ is a quadratic form over $F$.

A generalized quadratic form $\rho$ on $V$ is called  \emph{isotropic} if $\rho(v)=0$ for some
nonzero $v\in V$, and \emph{anisotropic} otherwise.
\medskip

Let $L/F$ be a finite extension such that  $D_L:=D\otimes_F L$ is a division algebra (for example, this holds when $[L:F]$ is coprime to the degree of $D$).  Set $\theta_L:=\theta\otimes\id_L$.
Since $
\alt_\lambda(D_L,\theta_L)=\alt_\lambda(D,\theta)\otimes_F L$,
we have
\[
D_L/\alt_\lambda(D_L,\theta_L)=\overline D\otimes_F L=:\overline D_L.
\]
There exists a unique generalized quadratic form
$
\rho_L:V_L\to\overline D_L
$
satisfying $\rho_L(x\otimes\alpha)=\alpha^2\rho(x)$ for all $x\in V$ and
$\alpha\in L$, whose polar map $h_{\rho_L}$ is the $L$-bilinear extension of $h_\rho$, i.e.,
$h_{\rho_L}(x\otimes\alpha,y\otimes\beta)=\alpha\beta h_\rho(x,y)$ for all $x,y\in V$ and $\alpha,\beta\in L$.

The next theorem holds in arbitrary characteristic and follows from \cite[Thm.~17.14]{ekm} and  \cite[Cor.~18.5]{ekm}.
\begin{theorem}\label{ABodd}
Let $L/F$ be a finite field extension of odd degree, and let $q_1,q_2$ be
quadratic forms on an $F$-vector space $V$.
If $(q_1)_L$ and $(q_2)_L$ have a common nonzero isotropic vector over $L$, then $q_1$ and $q_2$ have one over $F$.
\end{theorem}

Suppose that $D\neq F$ and $\theta$ is symplectic. 
If $\car F\neq2$, choose a unit $u\in D$ with $\theta(u)=-u$.
If $\car F=2$, let $u\in\sym(D,\theta)\setminus\alt(D,\theta)$ be a unit (see \cite[(2.6)]{knus}).
Define
$
\sigma=\Int(u)\circ\theta,
$
where $\Int(u)$ is the inner automorphism of $D$ induced by $u$.
By \cite[(2.7)]{knus}, the map $\sigma$ is an orthogonal involution on $D$.

Now let $(h,\rho)$ be a generalized quadratic form over $(D,\theta)$. Define
\[
h'(x,y):=u h(x,y),
\qquad
\rho'(x):=u \rho(x).
\]
Then $h'$ is a $\lambda'$-hermitian form over $(D,\sigma)$, where $\lambda'=-\lambda$.
Accordingly, $\rho'(x)=u\rho(x)$ is well-defined as an element of
$D/\alt_{\lambda'}(D,\sigma)$, because
\begin{equation}\label{eq2}
u\cdot \alt_\lambda(D,\theta)=\alt_{\lambda'}(D,\sigma),
\end{equation}
by \cite[(2.7)]{knus}.
It is readily seen that $(h',\rho')$ is a generalized quadratic form over
$(D,\sigma)$, of type $\lambda'$.
Note that \eqref{eq2} implies that $\rho$ is isotropic if and only if $\rho'$ is isotropic.
Moreover, for every field extension $L/F$ such that $D_L$ is a division algebra, the involution
$\sigma_L=\Int(u)\circ\theta_L$ is orthogonal, \eqref{eq2} holds over $L$, and
$(\rho')_L=(\rho_L)'$; hence $\rho_L$ is isotropic if and only if $(\rho')_L$ is isotropic.
Consequently, as far as the Artin--Springer property is concerned, we may
replace $(\theta,\lambda,h,\rho)$ by $(\sigma,\lambda',h',\rho')$, and thus
consider the case where $\theta$ is orthogonal.

A \emph{quaternion algebra} over $F$ is a central simple algebra of dimension
$4$.
Every quaternion algebra $Q$ has a basis
$(1,i,j,k)$ with $i^2,j^2\in F^\times$ and $ij=k=-ji$ if $\operatorname{char}F\ne2$; and 
$i^2+i\in F$, $j^2\in F^\times$ and $ij=k=ji+j$ if $\operatorname{char}F=2$ (see \cite[p. 25]{knus}).
Such a basis is called a \emph{quaternion basis} of $Q$.

Let $(Q,\theta)$ be a quaternion division algebra with orthogonal involution
over $F$ and let
$\rho$ be a generalized quadratic form of type $\lambda$ over $(Q,\theta)$.
If $\car F\neq2$ and $\lambda=-1$, then
$\alt_{-1}(Q,\theta)=\sym(Q,\theta)$ is three-dimensional by
\cite[(2.6)]{knus}, hence $\overline Q$ is one-dimensional. Identifying $\overline Q$ with $F$,
the map $\rho$ becomes a quadratic form on the $F$-vector space $V$, whose scalar extension to $L$
is $\rho_L$; so Theorem \ref{main} reduces to the classical Artin--Springer
theorem for quadratic forms \cite[Cor.~18.5]{ekm}.
Hence we will consider the case $\lambda=1$ in the sequel.

We now fix some notation for the rest of the paper.
Let $(Q,\theta)$ be a quaternion division algebra with orthogonal involution
over $F$.
Let $V$ be a finite-dimensional right vector space over $Q$, and let
$\rho\colon V\to\overline Q$ be a generalized quadratic form over $(Q,\theta)$.
According to \cite[(1.1)]{et}, the hermitian form $h$ in
\eqref{eq1} is uniquely determined by $\rho$ (the argument of \cite{et} also works in characteristic different from $2$).
We call this form \textit{the polar form} of $\rho$ and we denote it $h_\rho$.
Also, for every $x\in V$, we have $h_\rho(x,x)=a+\theta(a)$
for any representative $a\in Q$ of $\rho(x)$.
We also fix a quaternion basis $(1,i,j,k)$ of $Q$ with $\theta(i)=i$ and $\theta(j)=j$.
This is always possible by a simple argument (see \cite[(6.1)]{m} and \cite[(3.3)]{me}).

%======================================================================
\section{The main result}\label{sec:split}

Let $(Q,\theta)$, $(V,\rho)$ and the quaternion basis $(1,i,j,k)$ of $Q$
be as above. We work in one of the following two cases.

\begin{itemize}
\item[(I)] $\operatorname{char}F\neq2$. The elements $i,j\in Q$ satisfy
$i^2,j^2\in F^\times$ and $ij=k=-ji$. Set
$K=F(i)$, $\eta=j^2$, $\varepsilon_0=j$ and $\delta=2\eta$.
Then $\alt(Q,\theta)=Fk$ and $\overline Q=\overline{F+Fi+F\varepsilon_0}$.

\item[(II)] $\operatorname{char}F=2$. The elements $i,j\in Q$ satisfy
$i^2+i\in F$, $j^2\in F^\times$ and $ij=k=ji+j$. Set
$K=F(i)$, $\eta=j^2$, $\varepsilon_0=k$ and $\delta=\eta$.
Then $\alt(Q,\theta)=Fj$ and $\overline Q=\overline{F+Fi+F\varepsilon_0}$.
\end{itemize}
In both cases $\eta\in F^\times$, $\delta\in F^\times$, $\theta(i)=i$ and $\theta(j)=j$. The field $K=F(i)$ is a separable
quadratic extension of $F$.
Let $\iota$ be the nontrivial automorphism of $K/F$ and let
$N(\alpha)=\alpha\iota(\alpha)$ be the norm of $\alpha\in K$ in $K/F$. Then
$\iota(i)=-i$ if $\car F\neq2$ and $\iota(i)=i+1$ if $\car F=2$.
In both cases, we have
\begin{itemize}
\item[(F1)] $Q=K\oplus Kj$, 
$j\alpha=\iota(\alpha)j$ and $\alpha j=j\iota(\alpha)$ for all $\alpha\in K$.

\item[(F2)] For $\xi=\alpha+\beta j$ with $\alpha,\beta\in K$, we have 
$\theta(\xi)=\alpha+\iota(\beta)j$.
\end{itemize}
We also have
\begin{itemize}
\item[(F3)] $\varepsilon_0+\theta(\varepsilon_0)=g_0$ with
$g_0=2j$ in case (I) and $g_0=j$ in case (II).
In particular, $g_0\in Q^\times$.
\end{itemize}
Let $\pi_1\colon\overline Q\to F$ and $\pi_2\colon\overline Q\to K$ be the
$F$-linear maps satisfying
\[\pi_1(\bar1)=\pi_1(\bar i)=0,\ \  \pi_1(\bar\varepsilon_0)=1,\ \
\pi_2(\bar1)=1, \ \ \pi_2(\bar i)=i\ \ {\rm and} \ \ \pi_2(\bar\varepsilon_0)=0.
\]
Then $\pi_1\oplus\pi_2\colon\overline Q\to F\oplus K$ is injective.
Define
\[
\phi=\pi_1\circ\rho\colon V\to F,\qquad \psi=\pi_2\circ\rho\colon V\to K.
\]
Write $\psi=\psi_1+i\psi_2$ with $\psi_1,\psi_2\colon V\to F$. Since $\rho(xs)=s^2\rho(x)$ for $s\in F$ and the polar map of $\rho$ is
$(x,y)\mapsto\overline{h_\rho(x,y)}$, which is $F$-bilinear, the maps $\phi,\psi_1,\psi_2$ are quadratic forms over $F$ on the $F$-vector space $V$.
By the injectivity of $\pi_1\oplus\pi_2$, for all $x\in V$,
\begin{equation}\label{zero}
\rho(x)=0\  {\rm if\ and\ only\ if} \ \phi(x)=0\  {\rm and} \ \psi(x)=0.
\end{equation}

For a field extension \(L/F\), set
$
\overline{Q}_L:=\overline{Q}\otimes_F L$ and $K_L:=K\otimes_F L$.
The scalar extensions
\[
(\pi_1)_L:=\pi_1\otimes\id_L:
\overline{Q}_L\rightarrow L
\quad {\rm and}\quad
(\pi_2)_L:=\pi_2\otimes\id_L:
\overline{Q}_L\rightarrow K_L,
\]
are well-defined. Define
\[
\phi_L:=(\pi_1)_L\circ\rho_L:
V_L\rightarrow L
\quad {\rm and}\quad
\psi_L:=(\pi_2)_L\circ\rho_L:
V_L\longrightarrow K_L.
\]
The map $\phi_L$ is a quadratic form on the $L$-vector space $V_L$ which coincides with $\phi$ on $V$
and whose polar form is the $L$-bilinear extension of that of $\phi$; hence $\phi_L$ is the
scalar extension of the quadratic form $\phi$ to $L$.
Similarly, writing
$
\psi=\psi_1+i\psi_2
$
with $\psi_1,\psi_2:V\to F$, we have
$
\psi_L=(\psi_1)_L+i(\psi_2)_L,
$
where \(i\) denotes \(i\otimes1\in K_L\) and $(\psi_1)_L,(\psi_2)_L$ are the scalar extensions of $\psi_1,\psi_2$ to $L$.
In particular, for field extensions $F\subseteq L\subseteq L'$, the restriction of $\phi_{L'}$
(resp. $\psi_{L'}$) to $V_L\subseteq V_{L'}$ is $\phi_L$ (resp. $\psi_L$).

We now fix $E=F(t)$, where $t$ is an indeterminate.
\begin{proposition}\label{prop:reduction}
Let $L/F$ be a finite extension of odd degree. If $\rho_L$ is isotropic,
then there is a nonzero $x\in V\otimes_F E$ with
$\phi_E(x)+t \psi_E(x)=0$.
\end{proposition}

\begin{proof}
Let $0\neq v\in V\otimes_F L$ with $\rho_L(v)=0$. Applying the $L$-linear
maps $(\pi_1)_{L},(\pi_2)_L$ gives
$\phi_L(v)=0$ and $\psi_L(v)=0$.
Since $K_L$ has $L$-basis $1,i$, the equality
$
\psi_L(v)=0$
implies
$(\psi_1)_L(v)=(\psi_2)_L(v)=0$.
Consider the two quadratic forms over $E$ on $V\otimes_F E$ defined by
$q_1=\phi_{E}+t(\psi_1)_{E}$ and $q_2=(\psi_2)_{E}$.
Viewing $v$ as an element of $V\otimes_F L(t)$, we have
$(q_1)_{L(t)}(v)=(q_2)_{L(t)}(v)=0$.
Since
$
[L(t):E]=[L:F]
$
is odd, Theorem~\ref{ABodd} yields a nonzero
$x\in V\otimes_F E$ such that
$
q_1(x)=q_2(x)=0.
$
Therefore
$
\phi_E(x)+t\psi_E(x)
=q_1(x)+t i q_2(x)=0
$
in $K(t)$.
\end{proof}
Define
\[
\Lambda_t:=(\pi_1)_E+t (\pi_2)_E\colon\overline Q_E\to K_E,
\quad\text{so that}\quad
\Lambda_t(\rho_E(x))=\phi_E(x)+t \psi_E(x).
\]
For $\xi=\alpha+\beta j\in Q[t]\subseteq Q_E$ with $\alpha,\beta\in K[t]$, put
\[
G(\xi):=\Lambda_t\bigl(\overline{\theta(\xi) \varepsilon_0 \xi}\bigr)\in K[t].
\]

\begin{lemma}\label{G}
In both cases (I) and (II),
$
G(\alpha+\beta j)=N(\alpha)+\eta N(\beta)+\delta t \alpha \iota(\beta).
$
\end{lemma}

\begin{proof}
All computations take place in $Q[t]=K[t]\oplus K[t]j$ and use the fact that $t$ is central and fixed by $\iota$ and $\theta$.
We drop the subscript $E$ from $\pi_1$ and $\pi_2$ for simplicity.
Set $\alt=\alt(Q[t],\theta)$. The classes in
$\overline Q[t]$ in cases (I) and (II) are read off as follows, for
$P=P_0+P_1 i\in K[t]$,
$R=R_0+R_1 i\in K[t]$ and $S=P+Rj\in Q[t]$ ($P_0,P_1,R_0,R_1\in F[t]$):
\begin{itemize}
\item[(I)] $\overline{P+R j}=\bar P+R_0\bar j$
(as $R_1 i j=R_1 k\in\alt$), so $\pi_1(\bar S)=R_0$, $\pi_2(\bar S)=P$;
\item[(II)] $\overline{P+R j}=\bar P+R_1\bar k$
(as $R_1 i j=R_1 k$ and $R_0 j\in\alt$), so $\pi_1(\bar S)=R_1$, $\pi_2(\bar S)=P$.
\end{itemize}
\emph{Case (I).} Here $\varepsilon_0=j$. By (F1) and (F2), we have $\varepsilon_0\xi=\iota(\alpha)j+\eta\iota(\beta)$,
hence
\[
\theta(\xi)\varepsilon_0\xi
=(\alpha+\iota(\beta)j)\bigl(\iota(\alpha)j+\eta\iota(\beta)\bigr)
=2\eta \alpha\iota(\beta)+\bigl(N(\alpha)+\eta N(\beta)\bigr)j,
\]
so $\pi_1(\overline{\theta(\xi)\varepsilon_0\xi})=N(\alpha)+\eta N(\beta)$ and
$\pi_2(\overline{\theta(\xi)\varepsilon_0\xi})=2\eta \alpha\iota(\beta)$.
\\
\emph{Case (II).} Here $\varepsilon_0=k=ij$, so
$
\varepsilon_0\xi=i\iota(\alpha)j+\eta i\iota(\beta).
$
Hence, using (F1),
\[
\theta(\xi)\varepsilon_0\xi
=(\alpha+\iota(\beta)j)\bigl(i\iota(\alpha)j+\eta i\iota(\beta)\bigr)
=\eta\bigl(i+\iota(i)\bigr)\alpha\iota(\beta)
+\bigl[iN(\alpha)+\eta\iota(i)N(\beta)\bigr]j .
\]
Since $\iota(i)=1+i$, we have $i+\iota(i)=1$, and the bracket equals
$
\eta N(\beta)+\bigl(N(\alpha)+\eta N(\beta)\bigr)i,
$
hence
\[\theta(\xi)\varepsilon_0\xi
=\eta \alpha\iota(\beta)+\eta N(\beta)j+\bigl(N(\alpha)+\eta N(\beta)\bigr)k.\]
Thus $\pi_1(\overline{\theta(\xi)\varepsilon_0\xi})=N(\alpha)+\eta N(\beta)$ and
$\pi_2(\overline{\theta(\xi)\varepsilon_0\xi})=\eta \alpha\iota(\beta)$.

In each case $G(\xi)=\pi_1+t\pi_2$ is the stated expression, with
$\delta=2\eta$ in case (I) and $\delta=\eta$ in case (II).
\end{proof}

\begin{lemma}\label{Phi}
Define $\Phi\colon Q[t]\to Q[t]$ by
\[
\Phi(\alpha+\beta j)=\bigl(\iota(\alpha)+\delta t \iota(\beta)\bigr)-\iota(\beta) j
\qquad(\alpha,\beta\in K[t]).
\]
Then $G(\Phi(\xi))=G(\xi)$ for all $\xi\in Q[t]$, and
$\Phi(\xi)\ne0$ if $\xi\ne0$.
\end{lemma}

\begin{proof}
Let $\xi=\alpha+\beta j\in Q[t]$. Put
$\alpha'=\iota(\alpha)+\delta t\iota(\beta)$ and $\beta'=-\iota(\beta)$.
Then $\Phi(\xi)=\alpha'+\beta'j$. 
By Lemma~\ref{G},
\[
G(\Phi(\xi))=N(\alpha')+\eta N(\beta')+\delta t \alpha'\iota(\beta').
\]
We have $N(\beta')=N(\beta)$ and $\delta t \alpha'\iota(\beta')
=-\delta t \iota(\alpha)\beta-\delta^2t^2N(\beta)$.
Also, using $\iota(\alpha')=\alpha+\delta t\beta$ we obtain
\begin{align*}
N(\alpha')
=N(\alpha)+\delta t \iota(\alpha)\beta
+\delta t \alpha\iota(\beta)+\delta^2t^2N(\beta).
\end{align*}
Adding these gives
\[
N(\alpha')+\eta N(\beta')+\delta t \alpha'\iota(\beta')
=N(\alpha)+\delta t \alpha\iota(\beta)+\eta N(\beta).
\]
Hence $G(\Phi(\xi))=G(\xi)$.

Finally, if $\Phi(\xi)=0$, then $\alpha'=0$ and $\beta'=0$.
From $\beta'=-\iota(\beta)=0$ we get $\beta=0$, and then
$\alpha'=\iota(\alpha)=0$ gives $\alpha=0$. Thus $\xi=0$.
\end{proof}
\begin{remark}
One checks directly that \(\Phi^2=\mathrm{id}\); this confirms the fact that \(\Phi(\xi)\ne0\) whenever \(\xi\ne0\).
\end{remark}

\begin{theorem}\label{Q2}
If $\rho$ is anisotropic, then there is no nonzero
$x\in V\otimes_F E$ with $\phi_E(x)+t\psi_E(x)=0$.
\end{theorem}

\begin{proof}
Suppose $\rho$ is anisotropic and, by contradiction, that there is a nonzero
solution. Since
\[
\phi_E(xs)+t\psi_E(xs)=s^2\bigl(\phi_E(x)+t\psi_E(x)\bigr)
\]
for $s\in E$, we may clear denominators and assume $x\in V[t]$.
Choose a nonzero solution $x=\sum_{m=0}^d x_m t^m\in V[t]$ with $x_0,\dots,x_d\in V$ such that $d$ is minimal,
and put $z=x_d\ne0$.
Recall $\Lambda_t(\rho_E(x))=\phi_E(x)+t\psi_E(x)$.

Since $\rho_E(x)\in\overline Q[t]$ has degree $\le2d$ with
$t^{2d}$-coefficient $\rho(z)$, the
$t^{2d+1}$-coefficient of $\Lambda_t(\rho_E(x))$ is $\psi(z)$; hence
$\psi(z)=0$, i.e., $\pi_2(\rho(z))=0$.
If $\phi(z)=0$, then $\rho(z)=0$ by \eqref{zero}, contradicting
anisotropy. Hence,
$\mu:=\phi(z)\in F^\times$, i.e., $\pi_1(\rho(z))=\mu$.
The injectivity of $\pi_1\oplus\pi_2$ implies that $\rho(z)=\mu\bar\varepsilon_0$. If $d=0$, then
$\Lambda_t(\rho_E(x))=\Lambda_t(\rho(z))=\mu\ne0$, a contradiction; so $d\ge1$.

Set $g=h_\rho(z,z)$.
By \cite[(1.1)]{et}, we have $g=\mu (\varepsilon_0+\theta(\varepsilon_0))$, hence
(F3) implies that
$g=\mu g_0\in Q^\times$
is an invertible element, independent of $t$.
Put
\[
\zeta:=g^{-1}h_{\rho_E}(z,x)\in Q[t],\qquad y:=x-z \zeta .
\]
The $t^d$-coefficient of $h_{\rho_E}(z,x)$ is $g$.
Write $\zeta=\alpha+\beta j$ for some $\alpha,\beta\in K[t]$.
Since the $t^d$-coefficient of $\zeta$ is $1$
and $\zeta$ has degree $d$, the $t^d$-coefficient of $\alpha$ is $1$ and $\deg\beta\le d-1$.
Moreover, $h_{\rho_E}(z,y)=h_{\rho_E}(z,x)-g\zeta=0$, and $\deg y\le d-1$ because the
$t^d$-coefficients of $x$ and $z\zeta$ are both $z$.
Since $h_{\rho_E}(z\zeta,y)=\theta(\zeta)h_{\rho_E}(z,y)=0$, we obtain
\[
\rho_E(x)=\rho_E(z\zeta)+\rho_E(y)
=\theta(\zeta)\rho_E(z)\zeta+\rho_E(y)
=\mu \overline{\theta(\zeta)\varepsilon_0\zeta}+\rho_E(y).
\]
Applying $\Lambda_t$ gives
\begin{equation}\label{AGm}
\mu G(\zeta)=-\Lambda_t(\rho_E(y)).
\end{equation}
As $\deg\rho_E(y)\le2d-2$, the right-hand side has degree $\le2d-1$; hence
$\deg G(\zeta)\le2d-1$. By Lemma~\ref{G},
\[
G(\zeta)=N(\alpha)+\eta N(\beta)+\delta t\alpha\iota(\beta).
\]
Here $N(\alpha)$ has degree $2d$ with $t^{2d}$-coefficient $1$;
$\eta N(\beta)$ has degree $\le2d-2$; and $\alpha\iota(\beta)$ has degree
$\le2d-1$, with $t^{2d-1}$-coefficient $\iota(\beta_{d-1})$, where
$\beta_{d-1}$ is the $t^{d-1}$-coefficient of $\beta$.
So the $t^{2d}$-coefficient of $G(\zeta)$ is $1+\delta \iota(\beta_{d-1})$.
Therefore
\begin{equation}\label{top}
1+\delta \iota(\beta_{d-1})=0.
\end{equation}
Let $\zeta'=\Phi(\zeta)=\alpha'+\beta'j$, where $\alpha'=\iota(\alpha)+\delta t\iota(\beta)$ and $\beta'=-\iota(\beta)$.
The $t^d$-coefficient of $\alpha'$ is
$1+\delta\iota(\beta_{d-1})=0$ by \eqref{top}.
Hence, $\deg\alpha'\le d-1$ and $\deg\beta'\le d-1$, which implies that $\deg\zeta'\le d-1$.
Put $x':=z \zeta'+y$.
Then $\deg x'\le d-1$, and since $h_{\rho_E}(z,y)=0$, we obtain
$\rho_E(x')=\theta(\zeta')\rho_E(z)\zeta'+\rho_E(y)$, hence
\[
\phi_E(x')+t\psi_E(x')=\Lambda_t(\rho_E(x'))=\mu G(\zeta')+\Lambda_t(\rho_E(y))
=\mu G(\zeta)+\Lambda_t(\rho_E(y))=0,
\]
by Lemma~\ref{Phi} and \eqref{AGm}.
Finally $h_{\rho_E}(z,x')=g \zeta'\ne0$, because $g$ is invertible and
$\zeta'\ne0$ (Lemma~\ref{Phi}, as $\zeta\neq0$); so $x'\ne0$.
Thus $x'$ is a nonzero solution of degree $<d$, contradicting the
minimality of $d$.
\end{proof}

\begin{remark}
Since \(\Phi^2=\mathrm{id}\), the transformation \(\zeta\mapsto\zeta'=\Phi(\zeta)\) used above pairs up the solutions of \(G=c\) in an involutive way -- a Vieta-jump phenomenon in the spirit of Rost's degree-descent proof for conics \cite{rost}.
\end{remark}
We are now ready to prove the main result.
\begin{proof}[Proof of Theorem~\ref{main}]
By the reduction in
Section~\ref{sec:p}, we may assume that $\theta$ is orthogonal and $\lambda=1$, so we are
in one of the cases (I)--(II). Assume that $\rho_L$ is isotropic. By
Proposition~\ref{prop:reduction} there is a nonzero
$x\in V\otimes_F E$ with
$
\phi_E(x)+t\psi_E(x)=0.
$
By Theorem~\ref{Q2}, $\rho$ is isotropic.
\end{proof}

\end{document}